\documentclass{birkjour}
\usepackage{amsmath, amsthm, amsfonts, amssymb}
\usepackage{mathrsfs}
\usepackage{txfonts}

\newtheorem{theorem}{Theorem}[section]
\newtheorem{definition}{Definition}[section]
\newtheorem{coro}{Corollary}[section]
\newtheorem{lemma}{Lemma}[section]
\newtheorem{prop}{Proposition}[section]
\theoremstyle{definition}
\newtheorem{remark}{Remark}[section]

\allowdisplaybreaks

\usepackage[numbers]{natbib}
\setcitestyle{open={},close={}}

\def \l {\left}
\def \r {\right}

\def \bR {\Bbb R}

\def \bB {\Bbb B}

\def\be{\begin{equation}}
\def\ee{\end{equation}}
\def\ba{\begin{array}}
	\def\ea{\end{array}}
\def\bd{\begin{definition}}
	\def\ed{\end{definition}}
\def\bt{\begin{theorem}}
	\def\et{\end{theorem}}
\def\bc{\begin{corollary}}
	\def\ec{\end{corollary}}
\def\bl{\begin{lemma}}
	\def\el{\end{lemma}}
\def\bdm{\begin{displaymath}}
\def\edm{\end{displaymath}}

\begin{document}

%
%
%
%
%
%
%
%
%

\title[Operators on Herz-type spaces]
 {Operators on Herz-type spaces associated with ball quasi-Banach function spaces}

\author[M. Wei]{Mingquan Wei$^*$\footnote {$^*$ Corresponding author}}

\address{%
School of Mathematics and Stastics, Xinyang Normal University\\
Xinyang 464000, China}

\email{weimingquan11@mails.ucas.ac.cn}

\author[D. Yan]{Dunyan Yan}

\address{%
	School of Mathematical Sciences, University of Chinese Academy of Sciences\\
	Beijing 100049, China}

\email{ydunyan@ucas.ac.cn}



\date{September 26, 2021}

\begin{abstract}
Let $\alpha\in\bR$, $0<p<\infty$ and $X$ be a ball quasi-Banach function space on $\bR^n$. In this article, we introduce the Herz-type space $\dot{K}^{\alpha,p}_X(\bR^n)$ associated with $X$.  We identify the dual space of   $\dot{K}^{\alpha,p}_X(\bR^n)$, by which the boundedness of Hardy-Littlewood maximal operator on $\dot{K}^{\alpha,p}_X(\bR^n)$ is proved. By using the extrapolation
theorem on ball quasi-Banach function spaces,  we establish the extrapolation theorem on Herz-type spaces associated with ball quasi-Banach function spaces. Applying our extrapolation theorem, the boundedness of singular integral operators with rough kernels and their commutators, parametric Marcinkiewicz integrals, and oscillatory singular integral operators on $\dot{K}^{\alpha,p}_X(\bR^n)$ is obtained. As examples, we give some concrete function spaces which are members of Herz-type spaces associated with ball quasi-Banach function spaces.
\end{abstract}

\maketitle
\section{Introduction}
As natural generalizations of Lebesgue spaces, Herz spaces and their related function spaces have been playing
an important role in harmonic analysis 
and partial diferential equations, see [\cite{fu2007characterization,komori2004notes,liu2000boundedness,lu1996Hardy,miyachi2001remarks,wang2018higher,wei2021extrapolation}] for example. The Herz
spaces were initially introduced by Herz [\cite{herz1968lipschitz}] to study the Fourier series and Fourier transform in 1968. Nowadays, the boundedness of many important operators in harmonic analysis such as Hardy-Littlewood maximal operator, fractional integral operators and singular integral operators, have been extended to Herz spaces [\cite{almeida2012maximal,grafakos1998bilinear,komori2004notes,li1996boundedness,lu1996Hardy,tang2007cbmo}]. We refer the readers to the book of Lu et al. [\cite{lu2008herz}] for a comprehensive explaination for Herz-type spaces.

In the last two decades, as the development of variable exponent Lebesgue
spaces [\cite{diening2011lebesgue}], Herz spaces have been extended to the variable setting.  Herz spaces with variable exponent were introduced by Izuki [\cite{izuki2010boundedness}], who considered the boundedness of some sublinear operators on these spaces. In addition, the real variable theory for Herz-type Hardy spaces with variable exponent was establish by Wang and Liu [\cite{wang2012herz}] in 2012. After that, Herz-type spaces with variable exponent became important function spaces in the study of harmonic analysis and partial diferential equations, see [\cite{abdalmonem2021intrinsic,dong2015herz,izuki2010boundedness,izuki2010commutators,izuki2016hardy,liu2000boundedness,ragusa2009homogeneous,ragusa2012parabolic,wang2016continuity,wang2021estimates,wang2020boundedness,xu2018variable}].

Recently, Herz spaces were also extended to mixed Herz spaces by the first author [\cite{wei-2021}]. In the same paper [\cite{wei-2021}], the author established a criterion for the boundedness of operators on mixed Herz spaces if these operators are bounded on some weighted Lebesgue spaces.
By using this criterion, the boundedness of many operators, such as Hardy-Littlewood maximal operator, Calder{\'o}n-Zygmund singular integral operators, fractional integral operators, fractional maximal operator and Parametric Marcinkiewicz integrals on mixed Herz spaces was obtained. 
See also [\cite{wei2021characterization}] for the boundedness of commutators of Hardy-type operators on mixed Herz spaces.

As is well known, Banach function spaces introduced in [\cite{bennett1988interpolation}] unify many function spaces such as Lebesgue spaces and Lorentz spaces. Recall that ball quai-Banach function spaces studied by Sawano et al. [\cite{sawano2017hardy}] are generalizations of Banach function spaces. Compared with Banach function spaces, ball Banach
function spaces contain more function spaces. For instance, Morrey spaces, mixed-norm Lebesgue spaces, weighted Lebesgue spaces, and Orlicz-slice spaces are all
ball quasi-Banach function spaces, which may not be Banach function spaces. See [\cite{chang2020littlewood,ho2019erdelyi,izuki2019john,wang2020applications,wang2021weak,yan2020intrinsic,zhang2021weak}] for more studies on ball Banach function spaces. 

Motivated by the results in above references, it is natural to ask whether we can find some Herz-type spaces which contain Herz spaces, Herz spaces with variable exponent and mixed Herz spaces as special cases. In this paper, we give an affirmative answer by introducing Herz-type spaces associated with ball quasi-Banach function spaces. Moreover, we will study the mapping properties of various integral operators on these spaces. Now, we explain the outline
of this paper. In Sect. 2, we give the definition and some properties
of Herz-type spaces associated with ball quasi-Banach function spaces. In Sect. 3, the authors establish the boundedness of Hardy-Littlewood maximal operator on Herz-type spaces associated with ball Banach function spaces. Subsequently, the extrapolation theorem for Herz-type spaces associated with ball quasi-Banach function spaces is build by using the extrapolation theorem on ball quasi-Banach function spaces in Sect. 4. Applying our extrapolation theorem, we can obtaine the boundedness of singular integral operators with rough kernels and their commutators, parametric Marcinkiewicz integrals,  and oscillatory singular integral operators on Herz-type spaces associated with ball quasi-Banach function spaces in Sect. 5. At last, we give some concrete examples of Herz-type spaces associated with ball quasi-Banach function spaces in Sect. 6.

Throughout the paper, we use the following notations.

For any $r>0$ and $x\in \bR^n$, let $B(x,r)=\{y: |y-x|<r\}$ be the ball centered at $x$ with radius $r$.  Let $\bB=\{B(x,r):x\in\bR^n, r>0\}$ be the set of all such balls. 
We use $\chi_E$ and $|E|$ to denote the characteristic function and the Lebesgue measure of a measurable set $E$.
Let $\mathcal{M}(E)$ be the class of Lebesgue measurable functions on $E$. For any quasi-Banach space $X$, the space $L^X_{\rm{loc}}(E)$ consists of all functions $f\in\mathcal{M}(E)$ such that $f\chi_F\in X$ for all compact subsets $F\subseteq E$.  We denote the set of all complex numbers, all non-negative integers and all integers by $\mathbb{C}$, $\mathbb{N}$ and $\mathbb{Z}$, respectively. For a non-negative function $w\in L^1_{\rm{loc}}(\bR^n)$ and $0<p<\infty$, the weighted Lebesgue space $L^p_w(\bR^n)$ consists of all $f\in \mathcal{M}(\bR^n)$ such that $\|f\|_{L^p_w}:=\left(\int_{\bR^n}|f(x)|^pw(x)dx\right)^{1/p}<\infty$. By $A\lesssim B$, we
mean that $A\leq CB$ for some constant $C>0$, and $A\sim B$ means that $A\lesssim B$ and $B\lesssim A$.

For all $k\in\mathbb{Z}$, let $B_k=B(0,2^k)$, $C_k=B_k\backslash B_{k-1}$. Denote $\chi_k=\chi_{C_k}$ for $k\in\mathbb{Z}$, $\tilde{\chi}_k=\chi_{C_k}$ if $k\in\mathbb{N}^+$ and $\tilde{\chi}_0=\chi_{B_0}$.

\section{Definitions and preliminaries}
In this section, we first recall the definitions and  some basic properties of Herz spaces and ball Banach function spaces, and then give the definition of Herz-type spaces associated with ball quasi-Banach function spaces.

Now we recall the definition of the classical Herz spaces.
\begin{definition}\label{0-b*}
	Let ~$\alpha\in \bR$, $0< p,q<\infty$.
	
	{\rm (i)} The homogeneous Herz space $\dot{K}^{\alpha,p}_q(\bR^n)$ is defined by
	\begin{eqnarray*}
		{\dot{K}^{\alpha,p}_q}(\bR^n):=
		\l\{f\in L^q_{\rm{loc}}(\bR^n\backslash\{0\}): \|f\|_{\dot{K}^{\alpha,p}_q}<\infty\r\},
	\end{eqnarray*}
	where
	\begin{eqnarray*}
		\|f\|_{\dot{K}^{\alpha,p}_q}:=
		\l\{\sum_{k\in\mathbb{Z}}2^{kp\alpha}\|f\chi_{k}\|^p_{L^q}\r\}^{1/p}.
	\end{eqnarray*}
If $p=\infty$ or $q=\infty$, then we have to make appropriate modifications.

	{\rm (ii)} The non-homogeneous Herz space $K^{\alpha,p}_q(\bR^n)$ is defined by
	\begin{eqnarray*}
		{K^{\alpha,p}_q}(\bR^n):=
		\l\{f\in L^q_{\rm{loc}}(\bR^n): \|f\|_{K^{\alpha,p}_q}<\infty\r\},
	\end{eqnarray*}
	where \begin{eqnarray*}
		\|f\|_{K^{\alpha,p}_q}:=
		\l\{\sum_{k\in\mathbb{N}}2^{kp\alpha}\|f\tilde{\chi}_k\|^p_{L^q}\r\}^{1/p}.
	\end{eqnarray*}
If $p=\infty$ or $q=\infty$, then we have to make appropriate modifications.
\end{definition}
It is obvious that both $\dot{K}^{\alpha,p}_q(\bR^n)$ and $K^{\alpha,p}_q(\bR^n)$ are quasi-Banach spaces, and if $p,q\geq1$, they are further Banach spaces. For more properties for Herz spaces, we refer the readers to [\cite{lu2008herz}].

To study Herz-type spaces associated with ball quasi-Banach function spaces, we need the definition of ball quasi-Banach function spaces.
\begin{definition}\label{D-2-1}
	A quasi-Banach space $X\subseteq \mathcal{M}(\bR^n)$ is called a ball quasi-Banach function space if
	it satisfies
	
	{\rm (i)} $\|f\|_X=0$ implies that $f=0$ almost everywhere;
	
	{\rm (ii)} $|g|\leq|f|$ almost everywhere implies that $\|g\|_X\leq\|f\|_X$;
	
	{\rm (iii)}  $0\leq f_m\uparrow f$ almost everywhere implies that $0\leq \|f_m\|_X\uparrow \|f\|_X$;
	
	{\rm (iv)} $B\in \bB$ implies that $\chi_B\in X$.
\end{definition}
Moreover, a ball quasi-Banach function space $X$ is called a ball Banach function space if the norm
of $X$ satisfies the triangle inequality: for any $f,g\in X$
\begin{equation*}\label{E-2-2}
\|f+g\|_{X}\leq \|g\|_{X}+\|g\|_{X},
\end{equation*}
and for any $B\in \bB$, there exists a positive constant $C(B)$, depending on $B$, such that, for any $f\in X$, $$\int_{B}f(x)dx\leq C(B)\|f\|_X.$$
\begin{remark}\label{R-2-1}
	{\rm (i)} Observe that, in Definition \ref{D-2-1}, if we replace any ball $B$ by any bounded
	measurable set $E$, the definitions are mutually equivalent.
	
	{\rm (ii)} One can see from the definition that every ball quasi-Banach function space is a quasi-Banach
	space, and the converse is not necessary to be true. 
\end{remark}

The following notion of the associate space of a ball Banach function space can be found, for
instance, in [\cite{bennett1988interpolation}, Chapter 1, Definitions 2.1 and 2.3].
\begin{definition}\label{D-2-2}
	For any ball Banach function space $X$, the associate space (also called the K{\"o}the
	dual) $X'$	is defined by setting
	\begin{equation}\label{E-2-1}
	X':=\{f\in  \mathcal{M}(\bR^n): \|f\|_{X'}=\sup_{g\in X,~\|g\|_X\leq1}\|fg\|_{L^1(\bR^n)}<\infty\},
	\end{equation}
	where $\|\cdot\|_{X'}$ is called the associate norm of $\|\cdot\|_{X}$.
\end{definition}
\begin{remark}\label{R-2-2}
	From [\cite{sawano2017hardy}, Proposition 2.3], we know that, if $X$ is a ball Banach function space, then
	its associate space $X'$ is also a ball Banach function space.
\end{remark}

The following lemma gives an equivalent norm of ball Banach function spaces by using the second associate norm, see [\cite{zhang2021weak}, Lemma 2.6].
\begin{lemma}\label{L-2-1}
	Let $X$ be a ball Banach function space. Then $X$ coincides with its second associate space $X''$. In other words, a function $f$ belongs to $X$ if and only if it belongs to $X''$ and, in that case, $\|f\|_{X}=\|f\|_{X''}$.
\end{lemma}
From Definition \ref{D-2-1} and (\ref{E-2-1}), one can obtain H{\"o}lder's inequality on ball Banach function spaces. See also [\cite{bennett1988interpolation}, Theorem 2.4] for the proof.
\begin{lemma}\label{L-2-2}
	Let $X$ be a ball Banach function space, and $X'$
	its associate space. If $f\in X $ and $g\in X'$, then $fg$ is integrable and
	\begin{eqnarray*}
		\int_{\bR^N}|f(x)g(x)|dx\leq \|f\|_{X}\|g\|_{X'}.
	\end{eqnarray*}
\end{lemma}

The Hardy-Littlewood maximal operator $M$ is defined by setting, for any $f\in L^1_{{\rm loc}}(\bR^n)$ 
and $x\in\bR^n$,
\begin{equation}\label{E-2-3}
M (f)(x):=\sup_{B}\frac{1}{|B|}\int_B|f(y)|dy,
\end{equation}
where the supremum is taken over all balls $B\in \bB$ containing $x$.

If we assume some boundedness results of the Hardy-Littlewood maximal operator $M$ on the ball Banach function space $X$, then the norm $\|\cdot\|_{X}$ has properties similar to the classical Muckenhoupt
weights. 
\begin{lemma}\label{L-2-4}
	Let $X$ be a ball Banach function space. Suppose that the Hardy-Littlewood maximal
	operator $M$ is bounded on $X'$.
	Then for all balls $B\in \mathbb{B}$, we have
	\begin{equation}\label{G-2-2}
	\|\chi_{B}\|_{X}\|\chi_{B}\|_{X'}\lesssim |B|
	\end{equation}	
\end{lemma}
\begin{lemma}\label{L-2-5} 
	Let $X$ be a ball Banach function space. Suppose that $M$ is bounded on the associate
	space $X'$. Then there exists a constant $0<\delta_X<1$, such that for all balls $B\subseteq \mathbb{B}$ and all measurable sets $E\subseteq B$, we have
	\begin{equation}\label{G-2-3}
	\frac{\|\chi_E\|_X}{\|\chi_B\|_X}\lesssim\l(\frac{|E|}{|B|}\r)^{\delta_X}.
	\end{equation}
\end{lemma}
One can find the proofs of Lemma \ref{L-2-4}  and Lemma \ref{L-2-5} in [\cite{izuki2016boundedness}]. Although these lemmas were proved only for Banach function spaces in [\cite{izuki2016boundedness}], the results also hold for ball Banach function spaces by checking the proofs carefully. Note that in the particular case $X=L^q(\bR^n)$ ($0<q<\infty$), we have $\delta_X=1/q$ from a direct computation.

We also need the notion of the convexity of ball quasi-Banach spaces in the following sections, see [\cite{sawano2017hardy}, Definition 2.6].
\begin{definition}\label{D-2-3}
	Let $X$ be a ball quasi-Banach function space and $p\in (0,\infty)$. The $p$-convexification
	$X^p$ of $X$ is defined by setting $X^p:=\{f\in  \mathcal{M}(\bR^n): |f|^p\in X\}$ equipped with the quasi-norm $\|f\|_{X^p}=\|f^p\|_{X}^{1/p}$.
\end{definition}

Combining Herz spaces with ball quasi-Banach function spaces, we can define the Herz-type spaces associated with ball quasi-Banach function spaces as follows.
\begin{definition}\label{D-2-1-1}
	Let ~$\alpha\in \bR$, $0< p\leq\infty$ and $X$ be a ball quasi-Banach function space.
	
	{\rm (i)} The homogeneous Herz-type space associated with ball quasi-Banach function spaces $\dot{K}^{\alpha,p}_X(\bR^n)$ is defined by
	\begin{eqnarray*}
		{\dot{K}^{\alpha,p}_X}(\bR^n):=
		\l\{f\in L^X_{\rm{loc}}(\bR^n\backslash\{0\}): \|f\|_{\dot{K}^{\alpha,p}_X}<\infty\r\},
	\end{eqnarray*}
	where
	\begin{eqnarray*}
		\|f\|_{\dot{K}^{\alpha,p}_X}:=
		\l\{\sum_{k\in\mathbb{Z}}2^{kp\alpha}\|f\chi_{k}\|^p_{X}\r\}^{{1}/{p}}.
	\end{eqnarray*}
If $p=\infty$, then we have to make appropriate modifications.

	{\rm (ii)} The non-homogeneous Herz-type space $K^{\alpha,p}_X(\bR^n)$ is defined by
	\begin{eqnarray*}
		{K^{\alpha,p}_X}(\bR^n):=
		\l\{f\in L^X_{\rm{loc}}(\bR^n): \|f\|_{K^{\alpha,p}_X}<\infty\r\},
	\end{eqnarray*}
	where \begin{eqnarray*}
		\|f\|_{K^{\alpha,p}_X}:=
		\l\{\sum_{k\in\mathbb{N}}2^{kp\alpha}\|f\tilde{\chi}_k\|^p_{X}\r\}^{{1}/{p}}.
	\end{eqnarray*}
If $p=\infty$, then we have to make appropriate modifications.
\end{definition}
One can check that $\dot{K}^{\alpha,p}_X(\bR^n)$ and $K^{\alpha,p}_X(\bR^n)$ are quasi-Banach spaces, and if  $p\geq1$ and $X$ is a ball Banach function space, then they are further Banach spaces.

Obviously, if $X=L^q(\bR^n)~(0<q<\infty)$, then $\dot{K}^{\alpha,p}_X(\bR^n)$ and $K^{\alpha,p}_X(\bR^n)$ reduce to the classical Herz spaces $\dot{K}^{\alpha,p}_q(\bR^n)$ and $K^{\alpha,p}_q(\bR^n)$ in Difinition \ref{0-b*}. Moreover, if $X=L^{q(\cdot)}(\bR^n)$,  variable exponent Lebesgue spaces, we recover the Herz spaces with variable exponent defined in [\cite{izuki2010boundedness}]. Similarly, by choosing different ball quasi-Banach function spaces, we can define some new Herz-type spaces. In the last section of this paper, some concrete Herz-type spaces are given as examples of Herz spaces associated with ball quasi-Banach function spaces.

The following proposition gives the conditions such that Herz-type spaces associated with ball quasi-Banach function spaces are also ball Banach function spaces.
\begin{prop}\label{p-7}
	Let $\alpha\in\bR$, $0<p\leq\infty$, and $X$ be a ball Banach function space such that $\alpha+n\delta_X>0$, where $\delta_X$ is the constant appeared in (\ref{G-2-3}). Then the Herz-type spaces $\dot{K}^{\alpha,p}_X(\bR^n)$ and $K^{\alpha,p}_X(\bR^n)$ are all ball quasi-Banach function spaces. If further $1\leq p<\infty$ and $\alpha<n\delta_{X'}$ where $\delta_{X'}$ is the constant appeared in (\ref{G-2-3}), then $\dot{K}^{\alpha,p}_X(\bR^n)$ and $K^{\alpha,p}_X(\bR^n)$ are ball Banach function spaces.
\end{prop}
\begin{proof}
	We only prove the results for $\dot{K}^{\alpha,p}_X(\bR^n)$, and the other one can be proved in a similar way. 
	
	To prove $\dot{K}^{\alpha,p}_X(\bR^n)$ is a ball quasi-Banach function space with $0<p\leq\infty$ and $\alpha+n\delta_X>0$, we just need to show $\chi_{B_l}\in \dot{K}^{\alpha,p}_X(\bR^n)$ for any $l\in\mathbb{Z}$, since the other conditions are easy to verify. 
	
	Without loss of generality, we may assume $l>0$. When $0<p<\infty$, by using Lemma \ref{L-2-5} and the fact $\alpha+n\delta_X>0$, we have
	\begin{align*}
		\|\chi_{B_l}\|_{\dot{K}^{\alpha,p}_{X}}=&\l\{\sum_{k\in\mathbb{Z}}2^{kp\alpha}\|\chi_{B_l}\chi_{k}\|^{p}_{X}\r\}^{1/p}
		=\l\{\sum_{k=-\infty}^{l}2^{kp\alpha}\|\chi_{k}\|^{p}_{X}\r\}^{1/p}\\
		\lesssim&\l\{\sum_{k=-\infty}^{-1}2^{kp\alpha}\|\chi_{k}\|^{p}_{X}\r\}^{1/p}+\l\{\sum_{k=0}^{l}2^{kp\alpha}\|\chi_{k}\|^{p}_{X}\r\}^{1/p}\\
		\lesssim& \l\{\sum_{k=-\infty}^{-1}2^{kp\alpha}2^{knp\delta_X}\|\chi_{0}\|^{p}_{X}\r\}^{1/p}+\|\chi_l\|_{X}\\
		\lesssim& \l(\frac{2^{-p(\alpha+n\delta_X)}}{1-2^{-p(\alpha+n\delta_X)}}\r)^{1/p}\|\chi_0\|_{X}
		+\|\chi_l\|_{X}<\infty.
	\end{align*}
	 When $p=\infty$, in view of Lemma \ref{L-2-5}, we get
	 	\begin{align*}
	 \|\chi_{B_l}\|_{\dot{K}^{\alpha,p}_{X}}=&\sup_{k\in\mathbb{Z}}2^{k\alpha}\|\chi_{B_l}\chi_{k}\|_{X}
	 =\sup_{k\leq l}2^{k\alpha}\|\chi_{k}\|_{X}\\
	  \leq&\sup_{k\leq l}2^{k\alpha}\frac{\|\chi_{B_k}\|_{X}}{\|\chi_{B_l}\|_{X}}{\|\chi_{B_l}\|_{X}}\lesssim 2^{k\alpha} 2^{(k-l)n\delta_X}\|\chi_{B_l}\|_{X}\\
	 \lesssim& 2^{-ln\delta_X}\|\chi_{B_l}\|_{X}\sup_{k\leq l}2^{k(\alpha+n\delta_X)}\sim 2^{l\alpha}\|\chi_{B_l}\|_{X}<\infty,
	 \end{align*}
	 since $\alpha+n\delta_X>0$.
	 
	If further $1\leq p<\infty$, then $\dot{K}^{\alpha,p}_X(\bR^n)$ is a Banach space. 
	To show  $\dot{K}^{\alpha,p}_X(\bR^n)$ is a ball quasi-Banach function space with $1\leq p<\infty$ and $-n\delta_{X}<\alpha<n\delta_{X'}$, we also need to show  $\int_{B_{l}}|f(x)|dx\leq C(B_{l})\|f\|_X$ for any $l\in \mathbb{Z}$
	and $f\in \dot{K}^{\alpha,p}_X(\bR^n)$. 
	We now consider two different cases.
	
	Case I: $p=1$. For any $l\in \mathbb{Z}$, we have $\chi_{B_l}\in \dot{K}^{-\alpha,1}_{X'}(\bR^n)$ since $-\alpha+n\delta_{X'}>0$. This observation implies that 
		\begin{eqnarray*} 
			\sum_{k=\infty}^l 2^{-k\alpha}\|\chi_k\|_{X'}\leq \|\chi_{B_l}\|_{\dot{K}^{-\alpha,1}_{X'}}
				\end{eqnarray*} 
	for any fixed $l\in \mathbb{Z}$. Therefore, we obtain for $-\infty<k\leq l$,
	\begin{eqnarray}\label{G-jia}
		\|\chi_k\|_{X'}\leq 2^{k\alpha} \|\chi_{B_l}\|_{\dot{K}^{-\alpha,1}_{X'}},
	\end{eqnarray} 
	where $l$ is any fixed integer.
	
	By using (\ref{G-jia}) and H{\"o}lder's inequality on ball Banach function spaces, for any $l\in \mathbb{Z}$, we get
	\begin{align*} 
		\int_{B_{l}}|f(x)|dx=&\int_{\bR^n}|f(x)|\chi_{B_l}(x)dx
		=\sum_{k\in\mathbb{Z}}\int_{\chi_k}|f(x)|\chi_{B_l}(x)dx\\
		=& \sum_{k=-\infty}^l\int_{\chi_k}|f(x)|\chi_{B_l}(x)dx
		\leq \sum_{k=-\infty}^l\|f\chi_k\|_{X}\|\chi_k\|_{X'}\\
		\leq& 
		\sum_{k=-\infty}^l2^{k\alpha}\|f\chi_k\|_{X}\times \|\chi_{B_l}\|_{\dot{K}^{-\alpha,1}_{X'}}\\
		\leq& 
		\|\chi_{{B_l}}\|_{\dot{K}^{-\alpha,1}_{X'}}\|f\|_{\dot{K}^{\alpha,1}_{X}},
	\end{align*} 
	
	Case II: $1<p<\infty$. In this case,
	 let $p'$ satisfy $1/p+1/p'=1$. By using H{\"o}lder's inequality, for any $l\in \mathbb{Z}$, we have 
	\begin{align*} 
		\int_{B_{l}}|f(x)|dx\leq& \sum_{k=-\infty}^l\|f\chi_k\|_{X}\|\chi_k\|_{X'}\\
		\leq&
		 \l(\sum_{k=-\infty}^l2^{kp\alpha}\|f\chi_k\|^p_{X}\r)^{1/p}\l(\sum_{k=-\infty}^l2^{-kp'\alpha}\|\chi_k\|^{p'}_{X'}\r)^{1/p'}\\
		 \leq& 
		 \|\chi_{{B_l}}\|_{\dot{K}^{-\alpha,p'}_{X'}}\|f\|_{\dot{K}^{\alpha,p}_{X}},
		\end{align*} 
	where we have used the fact $\|\chi_{B_l}\|_{\dot{K}^{-\alpha,p'}_{X'}}<\infty$ since $-\alpha+n\delta_{X'}>0$.
	
	The proof is finished by combining the above two cases.
\end{proof}

\begin{remark}\label{R-2-4}
	In the following sections, we will only consider the homogeneous Herz-type spaces associated with ball Bananch funcition spaces since all the results obtained in this paper are still true for non-homogeneous Herz-type spaces associated with ball Bananch funcition spaces.
\end{remark}

\section{Boundedness of Hardy-Littlewood maximal oeprator on $\dot{K}^{\alpha,p}_X(\bR^n)$}
Our main result in this section is to establish the boundedness of Hardy-Littlewood maximal operator $M$ on $\dot{K}^{\alpha,p}_{X}(\bR^n)$. In order to show $M$ is well defined on $\dot{K}^{\alpha,p}_{X}(\bR^n)$, we first establish the dual spaces of $\dot{K}^{\alpha,p}_{X}(\bR^n)$.

As we know, the dual space of the classical Herz space $\dot{K}^{\alpha,p}_q(\bR^n)$ is also a Herz space $\dot{K}^{-\alpha,p'}_{q'}(\bR^n)$ [\cite{hernandez1999interpolation}], see also [\cite{wei-2021,wei2021extrapolation}] for the dual spaces of mixed Herz spaces and product Herz spaces. Inspired by [\cite{hernandez1999interpolation,wei-2021,wei2021extrapolation}], we will show that the dual space of $\dot{K}^{\alpha,p}_{X}(\bR^n)$ is also a Herz-type space associated with ball Banach function spaces under some mild conditions. To do so, we borrow some ideas from Hernandez and Yang [\cite{hernandez1999interpolation}].

Let $\mathbb{\dot{A}}=\{A_{j}\}_{j=-\infty}^\infty$ and $\mathbb{A}=\{A_{j}\}_{j=0}^\infty$, where $A_{j}$ are Banach spaces. For $\alpha\in \bR$ and
$0<p\leq \infty$, we define
$$\dot{l}^{\alpha}_p(\mathbb{\dot{A}})=\l\{a: a=\{a_{j}\}_{j=-\infty}^\infty, a_{j}\in A_{j}, 
\l(\sum_{j=-\infty}^\infty\l(2^{j\alpha}\|a_{j}\|_{A_{j}}\r)^p\r)^{1/p}<\infty\r\},$$
and
$$l^{\alpha}_p(\mathbb{A})=\l\{a: a=\{a_{j}\}_{j=0}^\infty, a_{j}\in A_{j}, 
\l(\sum_{j=0}^\infty\l(2^{j\alpha}\|a_{j}\|_{A_{j}}\r)^p\r)^{1/p}<\infty\r\},$$
where the usual modifications are made when $p=\infty$. Moreover,
$$\dot{C}^{\alpha}_0=\l\{a: a\in \dot{l}^{\alpha}_{\infty}(\mathbb{\dot{A}}), ~and~ 2^{j\alpha}\|a_{j}\|_{A_{j}}\rightarrow 0~ as~ |j|\rightarrow\infty\r\},$$
and
$$C^{\alpha}_0=\l\{a: a\in l^{\alpha}_{\infty}(\mathbb{A}), ~and~ 2^{j\alpha}\|a_{j}\|_{A_{j}}\rightarrow 0~ as~ j\rightarrow\infty\r\}.$$
Let $X^*$ denote the dual space of $X$, $\mathbb{\dot{A}}^*=\l\{A^*_{j}\r\}_{j=-\infty}^\infty$ and
$\mathbb{A}^*=\l\{A^*_{j}\r\}_{j=0}^\infty$.

For $0<p\leq\infty$, we denote $p'$ here satisfies $1/p+1/p'=1$ if $1<p\leq\infty$ and $p'=\infty$ if $0<p\leq1$. For all $f=\{f_{j}\}_{j=-\infty}^\infty\in \dot{l}^{-\alpha}_{p'}(\mathbb{\dot{A}^*})$ and 
$a=\{a_{j}\}_{j=-\infty}^\infty\in \dot{l}^{\alpha}_{p}(\mathbb{\dot{A}})$, we define 
$$f(a)=\sum_{j=-\infty}^\infty f_{j}(a_{j}),$$
which is a bounded linear functional over $\dot{l}^{\alpha}_{p}(\mathbb{\dot{A}})$ and satisfies 
$$\|f\|_{\dot{l}^{\alpha}_{p}(\mathbb{\dot{A}})}\leq\|f\|_{\dot{l}^{-\alpha}_{p'}(\mathbb{\dot{A}}^*)}.$$
Similar results hold for $l^{\alpha}_{p}(\mathbb{A})$ and $l^{-\alpha}_{p'}(\mathbb{A}^*)$.
Moreover, following [\cite{hernandez1999interpolation}] line by line, one can get the following similar result.
\begin{theorem}\label{abs}
	Let $\mathbb{\dot{A}}$ and $\mathbb{\dot{A}}^*$ be as above, $\alpha\in \bR$,  $0<p\leq\infty$ and  $1/p+1/p'=1$, where $p'=\infty$ if $0<p\leq1$. Then $\l(\dot{l}^{\alpha}_{p}(\mathbb{\dot{A}})\r)^*=\dot{l}^{-\alpha}_{p'}(\mathbb{\dot{A}}^*)$,  $\l(\dot{C}^{\alpha}_{0}(\mathbb{\dot{A}})\r)^*=\dot{l}^{-\alpha}_{1}(\mathbb{\dot{A}}^*)$,
	$\l(l^{\alpha}_{p}(\mathbb{A})\r)^*=l^{-\alpha}_{p'}(\mathbb{A}^*)$, and $\l(C^{\alpha}_{0}(\mathbb{A})\r)^*=l^{-\alpha}_{1}(\mathbb{A}^*)$.
\end{theorem}
We refer the reader to [\cite{hernandez1999interpolation}] for the details of the proof.

Before stating the dual results for Herz-type spaces associated with ball Banach function spaces, the following definition is needed.
\begin{definition}
	Let $\alpha\in\bR$ and $0<q\leq\infty$.
	Then a function $f\in L^{X}_{{\rm loc}}(\bR^n\backslash\{0\})$ is said to belong to the space $\l(\dot{K}^{-\alpha,\infty}_{X}\r)_0(\bR^n)$ if $f\in \dot{K}^{\alpha,\infty}_{X}(\bR^n)$ and $2^{k\alpha}\|f\chi_{k}\|_{X}\rightarrow0$ as $|k|\rightarrow\infty$.
	
\end{definition}
Theorem \ref{abs} gives us the dual spaces of Herz-type spaces associated with ball Banach function spaces.
\begin{theorem}\label{HD-1}
	Let $\alpha\in\bR, 0<p\leq\infty$ and $1/p+1/p'=1$, where $p'=\infty$ if $0<p\leq1$. Suppose $X$ is a ball Banach function space such that $X'=X^*$. Then 
	$$\l(\dot{K}^{\alpha,p}_X(\bR^n)\r)^*=\dot{K}^{-\alpha,p'}_{X'}(\bR^n),~~
	\l(\l(\dot{K}^{-\alpha,\infty}_{X}\r)_0(\bR^n)\r)^*=\dot{K}^{\alpha,1}_{X'}(\bR^n).$$
\end{theorem}
\begin{remark}\label{R-3-1}
	It is worthy mentioning that the condition $X'=X^*$ is satisfied for a large class of ball Banach function spaces. In fact, it was proved in [\cite{bennett1988interpolation}, Chapter 1, Corollary 4.3] that a Banach function space $X$ has an absolutely continuous norm if and only if $X'=X^*$. It is not hard to varify that this conclusion also holds for ball Banach function spaces. As is known to all, many function spaces, such as Lebesgue spaces, variable exponent Lebesgue spaces, mixed-norm Lebesgue spaces, Lorentz spaces and Orlicz spaces have absolutely continuous norms under some mild conditions, and therefore, they satisfy $X'=X^*$.
\end{remark}
Notice that if $1\leq p<\infty$, the space $\dot{K}^{\alpha,p}_X(\bR^n)$ is a Banach space. By using the closed-graph theorem, on can prove the following result.
\begin{coro}\label{c-dual}
	Let  $\alpha\in\bR, 1\leq p<\infty$, and $1/p+1/p'=1$. Suppose $X$ is a ball Banach function space such that $X'=X^*$. Then $f\in \dot{K}^{\alpha,p}_X(\bR^n)$ if and only if 
	$$\l|\int_{\bR^n} f(x)g(x)dx\r|<\infty$$
	for every $g\in \dot{K}^{-\alpha,p'}_{X'}(\bR^n)$, and in this case we have
	\begin{equation}
	\|f\|_{\dot{K}^{\alpha,p}_{X}}=\sup\l\{\l|\int_{\bR^n} f(x)g(x)dx\r|: \|g\|_{\dot{K}^{-\alpha,p'}_{X'}}\leq1\r\}.
	\end{equation}
\end{coro}
As a consequence Theorem \ref{HD-1}, we also have the following norm conjugate formula.
\begin{coro}\label{c-dual-1}
	Let  $\alpha\in\bR$, $0<p<\infty$ and $1/p+1/p'=1$, where $p'=\infty$ if $0<p\leq1$. 
	Suppose $X$ is a ball Banach function space such that $X'=X^*$. Then 
	$$\int_{\bR^n}|f(x)g(x)|dx\leq\|f\|_{{\dot{K}}^{\alpha,p}_{X}}
	\|g\|_{\dot{K}^{-\alpha,p'}_{X'}}$$
	for all $f\in \dot{K}^{\alpha,p}_{X}(\bR^n)$ and $g\in \dot{K}^{-\alpha,p'}_{X'}(\bR^n).$
\end{coro}

Now we show that $M$ is well defined on $\dot{K}^{\alpha,p}_{X}(\bR^n)$.
\begin{lemma}\label{LM-well}
	Let $\alpha\in\bR$, $0<p<\infty$, and $X$ be a ball Banach function space such that $X'=X^*$ and $\alpha<n\delta_{X'}$, where $\delta_{X'}$ is the constant appeared in (\ref{G-2-3}). Then we have $\dot{K}^{\alpha,p}_{X}(\bR^n)\subseteq L^1_{\rm{loc}}(\bR^n)$.
\end{lemma}
\begin{proof}
	From the proof of Proposition \ref{p-7}, we have 
	$\|\chi_{B_l}\|_{\dot{K}^{-\alpha,p'}_{X'}}<\infty$ since $-\alpha+n\delta_{X'}>0$, where $p'=\infty$ if $0<p\leq1$.
	
	Now for any $f\in\dot{K}^{\alpha,p}_{X}(\bR^n)$ and $l\in\mathbb{Z}$, by Corollary \ref{c-dual-1}, we have
	\begin{align*}
		\|f\chi_{B_l}\|_{L^1}\leq \|f\chi_{B_l}\|_{\dot{K}^{\alpha,p}_{X}}\|\chi_{B_l}\|_{\dot{K}^{-\alpha,p'}_{X'}}
		\leq \|f\|_{\dot{K}^{\alpha,p}_{X}}\|\chi_{B_l}\|_{\dot{K}^{-\alpha,p'}_{X'}}<\infty.
	\end{align*}
	As a consequence, we obtain $\dot{K}^{\alpha,p}_{X}(\bR^n)\subseteq L^1_{\rm{loc}}(\bR^n)$ by the arbitrariness of $l\in\mathbb{Z}$.
\end{proof}
The subsequent theorem gives the boundedness of $M$ on $\dot{K}^{\alpha,p}_{X}(\bR^n)$.
\begin{theorem}\label{T-3}
	Let  $\alpha\in\bR$, $0<p<\infty$, and $X$ be a ball Banach function space such that $X'=X^*$, the Hardy-Littlewood maximal operator $M$ is bounded on $X$ and $-n\delta_{X}<\alpha<n\delta_{X'}$, where $\delta_{X}$ and $\delta_{X'}$ are constants appeared in (\ref{G-2-3}). Then $M$ is bounded on $\dot{K}^{\alpha,p}_{X}(\bR^n)$.
\end{theorem}
\begin{proof}
	Due to $\dot{K}^{\alpha,p}_{X}(\bR^n)\subset L^1_{\rm{loc}}(\bR^n)$ by Lemma \ref{LM-well}, $M$ is well defined on $\dot{K}^{\alpha,p}_{X}(\bR^n)$.
	
	By definition, 
	\begin{align*}
		\|Mf\|_{\dot{K}^{\alpha,p}_{X}}=&\l\{\sum_{k=-\infty}^\infty2^{kp\alpha}\l\|\chi_{k}Mf\r\|^p_{X}\r\}^{1/p}\\
		\lesssim&\l\{\sum_{k=-\infty}^\infty2^{kp\alpha}
		\l\|\l(\sum_{l=-\infty}^\infty M(f\chi_{l})(x)\r)\chi_{k}(x)\r\|^p_{X}\r\}^{1/p}\\
		\lesssim&\l\{\sum_{k=-\infty}^\infty2^{kp\alpha}
		\l\|\l(\sum_{l=-\infty}^{k-2} M(f\chi_{l})(x)\r)\chi_{k}(x)\r\|^p_{X}\r\}^{1/p}\\
		&+\l\{\sum_{k=-\infty}^\infty2^{kp\alpha}
		\l\|\l(\sum_{l=k-1}^{k+1}M(f\chi_{l})(x)\r)\chi_{k}(x)\r\|^p_{X}\r\}^{1/p}\\
		&+\l\{\sum_{k=-\infty}^\infty2^{kp\alpha}
		\l\|\l(\sum_{l={k+2}}^{\infty}M(f\chi_{l})(x)\r)\chi_{k}(x)\r\|^p_{X}\r\}^{1/p}\\
		:=&I_1+I_2+I_3.
	\end{align*}	
	To finish the proof, we only need to show all the above terms are bounded by a constant-multiple of $\|f\|_{\dot{K}^{\alpha,p}_{X}}$.
	
	For the term $I_2$, by using the boundedness of $M$ on $X$, we obtain
	\begin{align*}
		I_2\lesssim&\l\{\sum_{k=-\infty}^\infty2^{kp\alpha}
		\l(\sum_{l={k-1}}^{k+1}\l\|M(f\chi_{l})\r\|^p_{X}\r)\r\}^{1/p}\\
		\lesssim&\l\{\sum_{k=-\infty}^\infty2^{kp\alpha}
		\l(\sum_{l={k-1}}^{k+1}\l\|f\chi_{l}\r\|^p_{X}\r)\r\}^{1/p}\\
		\lesssim&\l\{\sum_{k=-\infty}^\infty2^{kp\alpha}
		\l\|f\chi_{k}\r\|^p_{X}\r\}^{1/p}\sim \|f\|_{\dot{K}^{\alpha,p}_{X}}.
	\end{align*}
	For the term $I_1$, we use the fact that for $l\leq k-2$ and $x\in C_k$, there holds
	\begin{equation}\label{es-2}
	M(f\chi_{l})(x)\lesssim 2^{-kn}\|f\chi_{l}\|_{L^1}.
	\end{equation}
	From (\ref{es-2}) and H{\"o}lder's inequality on ball Banach funtion spaces, we deduce	
	\begin{align}\label{G-B-2}
	I_1\lesssim&\l\{\sum_{k-\infty}^\infty2^{kp\alpha}
	\l\|\l(\sum_{l=-\infty}^{k-2} \|f\chi_{l}\|_{L^1(\bR^n)}\r)2^{-kn}\chi_{k}(x)\r\|^p_{X}\r\}^{1/p}\nonumber\\
	=&\l\{\sum_{k=-\infty}^\infty2^{kp\alpha}
	\l(\sum_{l=-\infty}^{k-2}
	\|f\chi_{l}\|_{L^1(\bR^n)}\r)^p2^{-knp}\|\chi_k\|^p_{X}\r\}^{1/p}\nonumber\\
	\lesssim&\l\{\sum_{k=-\infty}^\infty2^{kp\alpha}
	\l(\sum_{l=-\infty}^{k-2} 2^{-kn}\|f\chi_{l}\|_{X}\|\chi_l\|_{X'}\|\chi_k\|_{X}\r)^p\r\}^{1/p}.
	\end{align}
	For  $l\leq k-2$, using Lemma \ref{L-2-4} and Lemma \ref{L-2-5} with $X'$, we have
	\begin{equation}\label{G-B-1}
	\|\chi_l\|_{X'}\|\chi_k\|_{X}\lesssim 2^{(l-k)n\delta_{X'}}\|\chi_k\|_{X'}\|\chi_k\|_{X}\lesssim 2^{(l-k)n\delta_{X'}}2^{kn}.
	\end{equation}
	Inserting  (\ref{G-B-1}) into (\ref{G-B-2}), we arrive at
	\begin{eqnarray}\label{G-B-3}
	I_1\lesssim\l\{\sum_{k=-\infty}^\infty2^{kp\alpha}
	\l(\sum_{l=-\infty}^{k-2} \|f\chi_{l}\|_{X}2^{(l-k)n\delta_{X'}}\r)^p\r\}^{1/p}.
	\end{eqnarray}
	Now we divide our discussion into two cases.	For $0<p\leq 1$, we have
	\begin{align*}
		I_1\lesssim&\l\{\sum_{k=-\infty}^\infty2^{kp\alpha}
		\l(\sum_{l=-\infty}^{k-2} \|f\chi_{l}\|_{X}2^{(l-k)n\delta_{X'}}\r)^p\r\}^{1/p}\\
		\lesssim& \l\{\sum_{k=-\infty}^\infty2^{kp\alpha}
		\l(\sum_{l=-\infty}^{k-2} \|f\chi_{l}\|^p_{X}2^{(l-k)np\delta_{X'}}\r)\r\}^{1/p}\\
		\lesssim& \l\{\sum_{l=-\infty}^\infty2^{lp\alpha}
		\|f\chi_{l}\|^p_{X}
		\sum_{k={l+2}}^{\infty}2^{{(k-l)}p(\alpha-n\delta_{X'})}\r\}^{1/p}\\
		\lesssim& \|f\|_{\dot{K}^{\alpha,p}_X},
	\end{align*}	
	since $\alpha<n\delta_{X'}$.
	
	For $1<p<\infty$, by using H{\"o}lder's inequality on ball Banach function spaces, we get
	\begin{align*}
		I_1\lesssim&\l\{\sum_{k=-\infty}^\infty2^{kp\alpha}
		\l(\sum_{l=-\infty}^{k-2} \|f\chi_{l}\|_{X}2^{(l-k)n\delta_{X'}}\r)^p\r\}^{1/p}\\
		\lesssim& \Bigg\{\sum_{k=-\infty}^\infty
		\l(\sum_{l=-\infty}^{k-2}2^{lp\alpha} \|f\chi_{l}\|^p_{X}2^{{(k-l)}(\alpha-n\delta_{X'})p/2}\r)\\
		&\times\l(\sum_{l=-\infty}^{k-2}2^{{(k-l)}(\alpha-n\delta_{X'})p'/2}\r)^{p/p'}
		\Bigg\}^{1/p}\\
		\lesssim& \Bigg\{\sum_{k=-\infty}^\infty
		\l(\sum_{l=-\infty}^{k-2}2^{lp\alpha} \|f\chi_{l}\|^p_{X}2^{{(k-l)}(\alpha-n\delta_{X'})p/2}\r)^{1/p}\\
		\lesssim& \l\{\sum_{l=-\infty}^\infty2^{lp\alpha}
		\|f\chi_{l}\|^p_{X}
		\sum_{k={l+2}}^{\infty}2^{{(k-l)}(\alpha-n\delta_{X'})p/2}\r\}^{1/p}\\
		\lesssim& \|f\|_{\dot{K}^{\alpha,p}_X}.
	\end{align*}	
	Therefore, we have $I_1\lesssim \|f\|_{\dot{K}^{\alpha,p}_X}$ noting that  $\alpha<n\delta_{X'}$.
	
	For $I_3$, when $l\geq k+2$, we have
	\begin{equation}\label{es-3}
	|M(f\chi_{l})(x)|~\chi_k(x)\lesssim 2^{-ln}\|f\chi_{l}\|_{L^1}\lesssim2^{-ln}\|f\chi_{l}\|_{X}\|\chi_{l}\|_{X'}.
	\end{equation}
	By using (\ref{es-3}) and H{\"o}lder's inequality on ball Banach function spaces, it follows that
	\begin{eqnarray}\label{es-4}
	I_3\lesssim\l\{\sum_{k=-\infty}^\infty2^{kp\alpha}
	\l(\sum_{l=k+2}^{\infty} 2^{-ln}\|f\chi_{l}\|_{X}\|\chi_{l}\|_{X'}\|\chi_{k}\|_{X}\r)^p\r\}^{1/p}.
	\end{eqnarray}
	For  $l\geq k+2$, using Lemma \ref{L-2-4} and Lemma \ref{L-2-5} with $X$, we have
	\begin{equation}\label{G-B-4}
	\|\chi_l\|_{X'}\|\chi_k\|_{X}\lesssim 2^{(k-l)n\delta_{X}}\|\chi_l\|_{X'}\|\chi_l\|_{X}\lesssim 2^{(k-l)n\delta_{X}}2^{ln}.
	\end{equation}
	Combining (\ref{es-4}) with (\ref{G-B-4}), it yields
	\begin{equation}\label{G-B-5}
	I_3\lesssim\l\{\sum_{k=-\infty}^\infty2^{kp\alpha}
	\l(\sum_{l=k+2}^{\infty} \|f\chi_{l}\|_{X}2^{(k-l)n\delta_{X}}\r)^p\r\}^{1/p}.
	\end{equation}
	Next we devide our discussion into two cases.
	
	When $0<p\leq1$, due to (\ref{G-B-5}), we have
	\begin{align*}
		I_3\lesssim&\l\{\sum_{k=-\infty}^\infty2^{kp\alpha}
		\sum_{l=k+2}^{\infty} \|f\chi_{l}\|^p_{X}2^{(k-l)pn\delta_{X}}\r\}^{1/p}\\
		\lesssim&
		\l\{\sum_{l=-\infty}^\infty2^{lp\alpha}\|f\chi_{l}\|^p_{X}
		\sum_{k=-\infty}^{l-2} 2^{(k-l)p(\alpha+n\delta_{X})}\r\}^{1/p}
		\lesssim\|f\|_{\dot{K}^{\alpha,p}_{X}},
	\end{align*}
	since $\alpha>-n\delta_{X}$.
	
	When $1<p<\infty$, by using H{\"o}lder's inequality and (\ref{G-B-5}), there holds
	\begin{align*}
		I_3\lesssim&\l\{\sum_{k=-\infty}^\infty2^{kp\alpha}
		\sum_{l=k+2}^{\infty} \|f\chi_{l}\|^p_{X}2^{(k-l)pn\delta_{X}}\r\}^{1/p}\\
		\lesssim&
		\l\{\sum_{k=-\infty}^\infty\l(\sum_{l=k+2}^{\infty}2^{lp\alpha}\|f\chi_{l}\|^p_{X}2^{\frac{k-l}{2}p(\alpha+n\delta_{X})}\r)
		\times \l(\sum_{l=k+2}^{\infty} 2^{\frac{k-l}{2}p'(\alpha+n\delta_{X})}\r)^{p/p'}\r\}^{1/p}\\
		\lesssim&
		\l\{\sum_{l=-\infty}^\infty2^{lp\alpha}\|f\chi_{l}\|^p_{X}
		\sum_{k=-\infty}^{l-2} 2^{\frac{k-l}{2}p(\alpha+n\delta_{X})}\r\}^{1/p}
		\lesssim\|f\|_{\dot{K}^{\alpha,p}_{X}},
	\end{align*}
	since $\alpha>-n\delta_{X}$.
	
	Therefore, we have $I_3\lesssim \|f\|_{\dot{K}^{\alpha,p}_{X}}.$	
	
	Combining the estimates of $I_1$, $I_2$ and $I_3$, we get the derised result.
\end{proof}

\begin{remark}
	For $0<p<\infty, 1<q<\infty$, it was proved in [\cite{li1996boundedness}] that the Hardy-Littlewood maximal operator $M$ is bounded on the classical Herz space $\dot{K}^{\alpha,p}_q(\bR^n)$ when $-{n}/{q}<\alpha< n/q'$. In view of $\delta_{L^q}=1/q$ and $\delta_{L^{q'}}=1/q'$, Theorem \ref{T-3} recover the result of [\cite{li1996boundedness}, Theorem 2.1] by setting $X=L^q(\bR^n)~(1<q<\infty)$.
\end{remark}
\section{Extrapolation on $\dot{K}^{\alpha,p}_X(\bR^n)$}
In this section, we establish the extrapolation theorem on Herz-type spaces associated with ball quasi-Banach function spaces by using the extrapolation theorem on ball quasi-Banach function spaces. 

To state the extrapolation on ball quasi-Banach function spaces, we need the Muckenhoupt weights. The classical Muckenhoupt weights play a key role in the study of weighted theory in harmonic analysis. We recall the definition of $A_p(\bR^n)$ weight (see [\cite{grafakos2008classical}]) briefly. 
\begin{definition}\label{D-2-4}
	For $1<p<\infty$, a locally integrable function $w:\bR^n\longrightarrow (0,\infty)$ is said to be an $A_p(\bR^n)$ weight if
	\begin{eqnarray*}
		[w]_{A_p}:=\sup_{B\in\mathbb{B}}\l(\frac{1}{|B|}\int_Bw(x)dx\r)\l(\frac{1}{|B|}\int_Bw(x)^{-\frac{p'}{p}}dx\r)^{\frac{p}{p'}}<\infty.
	\end{eqnarray*}
	A locally integrable function $w:\bR^n\longrightarrow (0,\infty)$ is said to be an $A_1(\bR^n)$ weight if for all $B\in\bB$,
	$$\frac{1}{|B|}\int_Bw(y)dy\leq Cw(x),~~~ {\rm a.e.}~ x\in B$$
	for some constant $C>0$. The infimum of all such $C$ is denoted by $[w]_{A_1}$. We denote $A_\infty(\bR^n)$ by the union of all $A_p(\bR^n)~(1\leq p<\infty)$ functions.
\end{definition}
Now we give the extrapolation theorem on ball quasi-Banach function spaces proved in [\cite{tao2021}, Proposition 2.14], see also [\cite{zhang2021weak}, Lemma 7.34].
\begin{lemma}\label{L-4-1}
	Let $X$ be a ball quasi-Banach function space such that $X^{1/s}$ is a ball Banach function space and $M$ is bounded on the associate space $(X^{1/s})'$ for some $0<s<\infty$. Let
	$T$ be an operator satisfying, for any given $w\in A_1(\bR^n)$ and any $f\in L^s_w(\bR^n)$,
	\begin{equation*}
	\|T(f)\|_{L^s_w}\lesssim \|f\|_{L^s_w},
	\end{equation*}
	where the implicated constant is independent on $f$, but depends on $s$ and $[w]_{A_1}$. Then for any $f\in X$,
	\begin{equation*}
	\|T(f)\|_{X}\lesssim \|f\|_{X}.
	\end{equation*}
\end{lemma}
Using Lemma \ref{L-4-1}, we can now give the extrapolation theorem on Herz-type spaces associated with ball quasi-Banach function spaces. 
\begin{theorem}\label{extrapo}
	Let  $\alpha\in\bR$, $0<p<\infty$, and $X$ be a ball quasi-Banach function space such that $X^{1/s}$ is a ball Banach function space and $M$ is bounded on the associate space $(X^{1/s})'$ for some $s<p$. Suppose $(X^{1/s})'=(X^{1/s})^*$ and $-n\delta_{X^{1/s}}<s\alpha<n\delta_{(X^{1/s})'}$, where $\delta_{X^{1/s}}$ and $\delta_{(X^{1/s})'}$ are constants appeared in (\ref{G-2-3}). Let
	$T$ be an operator satisfying, for any given $w\in A_1(\bR^n)$ and any $f\in L^s_w(\bR^n)$,
	\begin{equation*}
	\|T(f)\|_{L^s_w}\lesssim \|f\|_{L^s_w},
	\end{equation*}
	where the implicated constant is independent on $f$, but depends on $s$ and $[w]_{A_1}$. Then for  for any $f\in \dot{K}^{\alpha,p}_X(\bR^n)$,
	\begin{equation*}
	\|T(f)\|_{\dot{K}^{\alpha,p}_X}\lesssim \|f\|_{\dot{K}^{\alpha,p}_X}.
	\end{equation*}
\end{theorem}
\begin{proof}
	Note that $\dot{K}^{s\alpha,p/s}_{X^{1/s}}(\bR^n)$ is a ball Banach function space because $p/s>1$ and $X^{1/s}$ is a ball Banach function space. Since the condition $-n\delta_{X^{1/s}}<s\alpha<n\delta_{(X^{1/s})'}$ is equivalent to $-n\delta_{(X^{1/s})'}<-s\alpha<n\delta_{X^{1/s}}$, we obtain the boundedness of $M$ on $\dot{K}^{-s\alpha,(p/s)'}_{(X^{1/s})'}(\bR^n)$ by using Theorem \ref{T-3}.
	
	On the one hand, the convexity of ball quasi-Banach function spaces yields
	\begin{equation}\label{g-4-1}
	[\dot{K}^{\alpha,p}_{X}(\bR^n)]^{1/s}=\dot{K}^{s\alpha,p/s}_{X^{1/s}}(\bR^n).
	\end{equation}
	
	On the other hand, Theorem \ref{HD-1} guarantees that
	\begin{equation}\label{g-4-2}
	[\dot{K}^{s\alpha,p/s}_{X^{1/s}}(\bR^n)]^*=\dot{K}^{-s\alpha,(p/s)'}_{(X^{1/s})'}(\bR^n).
	\end{equation}
	Since $(X^{1/s})'=(X^{1/s})^*$, we know that the ball Banach function space $X^{1/s}$ has an absolutely continuous norm from Remark \ref{R-3-1}, which further yields that $[X^{1/s}]_b$ is dense in $X^{1/s}$ by using [\cite{bennett1988interpolation}, Chapter 1, Theorem 4.1 and Corollary 4.3]. Here and in what follows, for any ball Banach function space $Y$, $Y_b$ denotes the clousure in $Y$ of simple functions.
	
	Parallary to the proof of [\cite{ho2019extrapolation}, Proposition 4.2], we can prove that $[\dot{K}^{s\alpha,p/s}_{X^{1/s}}(\bR^n)]_b$ is dense in $\dot{K}^{s\alpha,p/s}_{X^{1/s}}(\bR^n)$, since the key point in the proof of [\cite{ho2019extrapolation}, Proposition 4.2] is the density of simple functions in $L^{p(\cdot)}(\bR^n)$. That is to say, $K^{s\alpha,p/s}_{X^{1/s}}(\bR^n)$ has an absolutely continuous norm by using [\cite{bennett1988interpolation}, Chapter 1, Theorem 3.11]. 
	
	Again by using [\cite{bennett1988interpolation}, Chapter 1, Corollary 4.3] and (\ref{g-4-1}), (\ref{g-4-2}), we have
	\begin{equation}\label{g-4-3}
	([\dot{K}^{\alpha,p}_{X}(\bR^n)]^{1/s})'=[\dot{K}^{s\alpha,p/s}_{X^{1/s}}(\bR^n)]^*=\dot{K}^{-s\alpha,(p/s)'}_{(X^{1/s})'}(\bR^n).
	\end{equation}
	The boundedness of  $M$ on $\dot{K}^{-s\alpha,(p/s)'}_{(X^{1/s})'}(\bR^n)$, together with (\ref{g-4-3}), yields that 
	the conditions of Lemma \ref{L-4-1} are satisfied for the ball quasi-Banach function spaces $\dot{K}^{\alpha,p}_{X}(\bR^n)$. 
	
	As a consequence of Lemma \ref{L-4-1}, we obtain the for all $f\in \dot{K}^{\alpha,p}_X(\bR^n)$,
	$$\|T(f)\|_{\dot{K}^{\alpha,p}_X}\lesssim \|f\|_{\dot{K}^{\alpha,p}_X}.$$
\end{proof}
The preceding theorem gives a complementary result for the extrapolation theory on Herz-type spaces. For the extrapolation theorems on some concrete Herz-type spaces, in particular, Herz-type spaces with variable exponent, the reader is refered to [\cite{ho2019extrapolation,ho2020spherical}].
\section{Boundedness of operators on $\dot{K}^{\alpha,p}_X(\bR^n)$}
In this section, with the help of Theorem \ref{extrapo}, we will study the boundedness of singular integral operators with rough kernels and their commutators, Parametric Marcinkiewicz integrals and oscillatory singular integral operators on ${\dot K}^{\alpha,p}_X(\bR^n)$. 
Other than the operators mentioned above, Theorem \ref{extrapo} can also apply to the boundededness of a large class of integral operators and their commutators as long as they are weighted bounded.

\subsection{Singular integral operators with rough kenels and their commutators}
In this subsection, we will study the boundedness of singular integral operators with rough kenels and theirs commutators on Herz-type spaces associated with ball quasi-Banach function spaces.

Suppose that $\mathbb{S}^{n-1}$ is the unit sphere in $\bR^n~(n\geq2)$ equipped with the normalized
Lebesgue measure $d\sigma$. We say that $\Omega$ is
a homogeneous function of degree zero if $\Omega(\lambda x)=\Omega(x)$, for any $\lambda>0$ and
$x\in\bR^n$. We say that $\Omega$ satisfies the vanishing moment condition if
\begin{equation}\label{van}
\int_{\mathbb{S}^{n-1}}\Omega(x')d\sigma(x')=0,
\end{equation}
where $x'=x/|x|$ for all $x\neq0$.

Let $\Omega\in L^1(\bR^n)$ be a homogeneous function satisfying the vanishing
moment condition. The singular integral with rough kernel $\Omega$ is defined by setting, for any $f\in L^1_{{\rm loc}}(\bR^n)$ and $x\in\bR^n$,
\begin{equation}\label{sin}
T_\Omega(f)(x):=\int_{\bR^n}\frac{\Omega(x-y)}{|x-y|^n}f(y)dy.
\end{equation}
The weighted norm inequality for singular integrals
with rough kernels can be stated as follows, see [\cite{duoandikoetxea1993weighted,watson1990weighted}] for the proof.
\begin{theorem}\label{wei-sin}
	Let $1<q<\infty$ and $\Omega\in L^q(\mathbb{S}^{n-1})$ be a homogeneous function
	of degree zero that satisfies the vanishing moment condition. If $p\in(q',\infty)$ and $w\in A_{p/q'}(\bR^n)$,
	then $T_\Omega$ is bounded on $L^p_w(\bR^n)$.
\end{theorem}
The preceding result and Theorem \ref{extrapo} yield the boundedness of $T_\Omega$ on $\dot{K}^{\alpha,p}_X(\bR^n)$.
\begin{theorem}\label{wei-singu}
	Let $1<q<\infty$ and $\Omega\in L^q(\mathbb{S}^{n-1})$ be a homogeneous function
	of degree zero that satisfies the vanishing moment condition. Suppose  $\alpha\in\bR$, $q'<p<\infty$, and $X$ is a ball quasi-Banach function space such that $X^{1/s}$ is a ball Banach function space and $M$ is bounded on the associate space $(X^{1/s})'$ for some $s$ satisfying $q'<s<p<\infty$. If $(X^{1/s})'=(X^{1/s})^*$ and $-n\delta_{X^{1/s}}<s\alpha<n\delta_{(X^{1/s})'}$, where $\delta_{X^{1/s}}$ and $\delta_{(X^{1/s})'}$ are constants appeared in (\ref{G-2-3}), then for all $f\in \dot{K}^{\alpha,p}_X(\bR^n)$, 
	$$\|T_{\Omega}(f)\|_{\dot{K}^{\alpha,p}_X}\lesssim\|f\|_{\dot{K}^{\alpha,p}_X}.$$
\end{theorem}

As is well known, commutators are also important operators in harmonic analysis. Recall that for a locally integrable function $b$ and an integral operator $T$, the commutator formed by $b$ and $T$ is defined by $[b,T]=bT-Tb$. 

The first significant result in this direction was made by Coifman et al. [\cite{coifman1976factorization}], which characterizes
the boundedness of such type commutators on the Lebesgue spaces via the
well-known space ${\rm BMO}(\bR^n)$. Recall that the space ${\rm BMO}(\bR^n)$, initially introduced by John and Nirenberg [\cite{john1961functions}], is defined to be the set of all locally integrable functions $f$ on $\bR^n$ such that
\begin{equation*}
\|f\|_{\rm BMO}:=\sup_{B\in\mathbb{B}}\frac{\|(f-f_B)\chi_B\|_{L^1}}{\|\chi_B\|_{L^1}}<\infty,
\end{equation*}
where
$f_B:=\frac{1}{|B|}\int_{B}f(x)dx$ is the average of $f$ over the ball $B$.

Now we focus on the boundedness of the commutator $[b,T_\Omega]$ on $\dot{K}^{\alpha,p}_X(\bR^n)$, where $b\in{\rm BMO}(\bR^n)$ and $T_\Omega$ is a singular integral with rough kernel defined in (\ref{sin}). We first recall the weighted boundedness of $[b,T_\Omega]$ established in [\cite{lu2007singular}].
\begin{theorem}\label{wei-sin-com}
	Let $1<q<\infty$ and $\Omega\in L^q(\mathbb{S}^{n-1})$ be a homogeneous function
	of degree zero that satisfies the vanishing moment condition. If $p\in(q',\infty)$ and $w\in A_{p/q'}(\bR^n)$ and $b\in{\rm BMO}(\bR^n)$,
	then $T_\Omega$ is bounded on $L^p_w(\bR^n)$.
\end{theorem}
Similarly, we can obtain the following result on the boundedness of $[b,T_\Omega]$ on $\dot{K}^{\alpha,p}_X(\bR^n)$ as a consequence of Theorem \ref{extrapo} and Theorem \ref{wei-sin-com}.
\begin{theorem}\label{wei-singu-com}
	Let $1<q<\infty$ and $\Omega\in L^q(\mathbb{S}^{n-1})$ be a homogeneous function
	of degree zero that satisfies the vanishing moment condition. Suppose  $\alpha\in\bR$, $q'<p<\infty$, and $X$ is a ball quasi-Banach function space such that $X^{1/s}$ is a ball Banach function space and $M$ is bounded on the associate space $(X^{1/s})'$ for some $s$ satisfying $q'<s<p<\infty$. If $b\in{\rm BMO}(\bR^n)$, $(X^{1/s})'=(X^{1/s})^*$ and $-n\delta_{X^{1/s}}<s\alpha<n\delta_{(X^{1/s})'}$, where $\delta_{X^{1/s}}$ and $\delta_{(X^{1/s})'}$ are constants appeared in (\ref{G-2-3}), then for all $f\in \dot{K}^{\alpha,p}_X(\bR^n)$, 
	$$\|[b,T_{\Omega}](f)\|_{\dot{K}^{\alpha,p}_X}\lesssim\|f\|_{\dot{K}^{\alpha,p}_X}.$$
\end{theorem}

\subsection{Parametric Marcinkiewicz integrals}
This subsection considers the mapping property of the parametric Marcinkiewicz integrals on Herz-type spaces associated with ball quasi-Banach function spaces.

We first recall the definition of parametric Marcinkiewicz integrals. 

Suppose $\Omega$ is
a homogeneous function of degree zero and satisfies the vanishing moment condition (\ref{van}).
For $0<\rho<n$, H{\"o}rmander [\cite{hormander1960estimates}] defined the parametric Marcinkiewicz integral operator $\mu^\rho_\Omega$ of higher dimension as 
$$\mu^\rho_{\Omega} (f)(x):=\l(\int_0^\infty|F^\rho_{\Omega,t}(x)|^2 \frac{dt}{t^{2\rho+1}} \r)^{1/2},$$
where
$$F^\rho_{\Omega,t}(x):=\int_{|x-y|<t}\frac{\Omega(x-y)}{|x-y|^{n-\rho}}f(y)dy.$$
Note that the operator $\mu^1_\Omega$ was first
introduced by Stein [\cite{stein1958functions}].

In [\cite{sato1998remarks}], Sato established the following weighted $L^p$ boundedness of $\mu^\rho_\Omega$ for all
$0<\rho<n$.
\begin{theorem}\label{T-Mar}
	Let $0<\rho<n$ and $\Omega\in L^\infty(\mathbb{S}^{n-1})$. If $w\in A_p$, $1<p<\infty$, then
	$$\|\mu^\rho_\Omega (f)\|_{L^p_w}\lesssim\|f\|_{L^p_w}.$$
\end{theorem}
Combining Theorem \ref{extrapo} with Theorem \ref{T-Mar}, the boundedness of $\mu^\rho_\Omega$ on Herz-type spaces associated with ball quasi-Banach function spaces can be stated as follows.
\begin{theorem}\label{T-Mar-1}
	Let $0<\rho<n$ and $\Omega\in L^\infty(\mathbb{S}^{n-1})$. Suppose  $\alpha\in\bR$, $1<p<\infty$, and $X$ is a ball quasi-Banach function space such that $X^{1/s}$ is a ball Banach function space and $M$ is bounded on the associate space $(X^{1/s})'$ for some $s$ satisfying $1<s<p<\infty$. If $(X^{1/s})'=(X^{1/s})^*$ and $-n\delta_{X^{1/s}}<s\alpha<n\delta_{(X^{1/s})'}$, where $\delta_{X^{1/s}}$ and $\delta_{(X^{1/s})'}$ are constants appeared in (\ref{G-2-3}), then for all $f\in \dot{K}^{\alpha,p}_X(\bR^n)$, we have
	$$\|\mu^\rho_\Omega (f)\|_{\dot{K}^{\alpha,p}_X}\lesssim\|f\|_{\dot{K}^{\alpha,p}_X}.$$
\end{theorem}
\subsection{Oscillatory singular integral operators}
This subsection is devoted to the boundedness of oscillatory singular integral operators on Herz-type spaces associated with ball quasi-Banach function spaces. First we recall the definition of oscillatory singular integral operators.

Let $K$ satisfy
\begin{equation}\label{Osc-1}
|K(x,y)|\lesssim\frac{1}{|x-y|^n},~~x\neq y,
\end{equation}
\begin{equation}\label{Osc-2}
|\nabla_xK(x,y)|+|\nabla_yK(x,y)|\lesssim\frac{1}{|x-y|^{n+1}},~~x\neq y.
\end{equation}

Let $P(x,y)$ be a real-valued polynomial on $\bR^n\times\bR^n$.
The oscillatory singular integral operator $T_{K,P}$ associated with $K$ and $P$ is defined as
\begin{equation*}
T_{K,P}(f)(x):={\rm p.v.}\int_{\bR^n\times\bR^n}e^{iP(x,y)}K(x,y)f(y)dy.
\end{equation*}

The study of oscillatory singular integral operators can be traced to the work of Ricci and Stein
[\cite{ricci1987harmonic}], Chanillo and Christ [\cite{chanillo1987weak}]. We refer the readers to [\cite{fan1995boundedness,lu1995oscillatory,luzhang1992weighted,pan1991hardy}] for more studies of oscillatory singular integral operators.
Particularly, the weighted norm inequalities for oscillatory singular integral
operators were obtained in [\cite{luzhang1992weighted}].
\begin{theorem}\label{Osc-wei}
	Let $1<p<\infty$, $P(x,y)$ be a real-valued polynomial on $\bR^n\times\bR^m$ and $K$ satisfies (\ref{Osc-1}) and (\ref{Osc-2}). If the Calder{\'o}n-Zygmund operator 
	\begin{equation}\label{OSc-3}
	T(f)(x):=\int_{\bR^n}K(x,y)f(y)dy
	\end{equation}
	is bounded on $L^2(\bR^n)$. Then for any $w\in A_p(\bR^n)$, $T_{P,K}$ is bounded on $L^p_w(\bR^n)$.
\end{theorem}
The above theorem, together with Theorem \ref{extrapo}, guarantees the boundedness of $T_{P,K}$ on Herz-type spaces associated with ball quasi-Banach function spaces.
\begin{theorem}\label{T-Mar-2}
	Let $1<p<\infty$, $P(x,y)$ be a real-valued polynomial on $\bR^n\times\bR^m$ and $K$ satisfies (\ref{Osc-1}) and (\ref{Osc-2}). Suppose  $\alpha\in\bR$ and $X$ is a ball quasi-Banach function space such that $X^{1/s}$ is a ball Banach function space and $M$ is bounded on the associate space $(X^{1/s})'$ for some $s$ satisfying $1<s<p<\infty$. If $(X^{1/s})'=(X^{1/s})^*$, $-n\delta_{X^{1/s}}<s\alpha<n\delta_{(X^{1/s})'}$, where $\delta_{X^{1/s}}$ and $\delta_{(X^{1/s})'}$ are constants appeared in (\ref{G-2-3}), and the Calder{\'o}n-Zygmund operator defined in (\ref{OSc-3}) is bounded on $L^2(\bR^n)$, then for all $f\in \dot{K}^{\alpha,p}_X(\bR^n)$, we have
	$$\|T_{K,P}(f)\|_{\dot{K}^{\alpha,p}_X}\lesssim\|f\|_{\dot{K}^{\alpha,p}_X}.$$
\end{theorem}
\section{Some examples}
In this section, we apply all above results to three concrete examples of Herz-type spaces associated with ball quasi-Banach function spaces, namely, Herz spaces with variable exponent, mixed Herz spaces and Herz-Lorentz spaces, respectively, in Subsections 6.1-6.3.
\subsection{Herz spaces with variable exponent}
We first recall some basic theories on variable exponent Lebesgue spaces. Let $q(\cdot):\bR^n\rightarrow (0,\infty)$ be a measurable function. Then the variable Lebesgue space
$L^{q(\cdot)}(\bR^n)$ is defined to be the set of all measurable functions $f$ on $\bR^n$ such that
\begin{eqnarray}
\|f\|_{L^{q(\cdot)}}:=\inf\l\{\lambda\in(0,\infty):\int_{\bR^n}[|f(x)|/\lambda]^{q(x)}dx\leq1\r\}<\infty.
\end{eqnarray}
It was pointed in [\cite{wang2021weak}] that  $L^{q(\cdot)}(\bR^n)$ is a ball quasi-Banach function space whenever $q(\cdot):\bR^n\rightarrow (0,\infty)$.
Moreover, if $q(x)\geq1$ for almost every $x\in\bR^n$, then $L^{q(\cdot)}(\bR^n)$ is a Banach function space (see,
for instance, [\cite{diening2009function}, Theorem 3.2.13]), and in this case, the associate space of $L^{q(\cdot)}(\bR^n)$ is the variable Lebesgue space $L^{q'(\cdot)}(\bR^n)$, where $q'(\cdot)$ satisfies ${1}/{q(\cdot)}+{1}/{q'(\cdot)}=1$ (see [\cite{cruz2013variable}, Proposition 2.37]).
We refer the reader to [\cite{diening2011lebesgue,kempka2014lorentz,yang2017survey}] for more
details on variable Lebesgue spaces.

By taking $X=L^{q(\cdot)}(\bR^n)$ in Definition \ref{D-2-1-1}, where $q(\cdot):\bR^n\rightarrow (0,\infty)$ is a measurable function, we obtain the homogeneous Herz space with variable exponent $\dot{K}^{\alpha,p}_{q(\cdot)}(\bR^n)$ and non-homogeneous Herz space with variable exponent ${K}^{\alpha,p}_{q(\cdot)}(\bR^n)$. 

For any measurable function $q(\cdot):\bR^n\rightarrow (0,\infty)$, let
$$q_{-}={\rm ess}\inf_{x\in\bR^n}q(x),~~~{\rm and}~~~q_{+}={\rm ess}\sup_{x\in\bR^n}q(x).$$

Then if $1\leq q_{-}\leq q_+<\infty$, the dual space of $L^{q(\cdot)}(\bR^n)$ is $L^{q'(\cdot)}(\bR^n)$ (see [\cite{cruz2013variable}, Theorem 2.80]).

A measurable function $q(\cdot):\bR^n\rightarrow (0,\infty)$ is said to be globally log-H{\"o}lder
continuous if there exists a $q_{\infty}>0$ such that, for any $x,y\in\bR^n$,
\begin{equation*}
|q(x)-q(y)|\lesssim\frac{1}{\log(e+1/|x-y|)}
\end{equation*}
and 
\begin{equation*}
|q(x)-q_{\infty}|\lesssim\frac{1}{\log(e+|x|)},
\end{equation*}
where the implicit positive constants are independent of $x,y\in\bR^n$.

The boundedness of Hardy-Littlewood maximal operator $M$ on $L^{q(\cdot)}(\bR^n)$ was obtained in [\cite{diening2011lebesgue}, Theorem 4.38]. For the readers' convenience, we state the result as follows.
\begin{lemma}\label{H-V-l}
	Let $\alpha\in\bR$, $0<p<\infty$ and $q(\cdot):\bR^n\rightarrow (0,\infty)$ be a globally log-H{\"o}lder continuous function satisfying
	$1<q_{-}<q_{+}<\infty$. then $M$ is bounded on $L^{q(\cdot)}(\bR^n)$.
\end{lemma}
By using Theorem \ref{T-3} and Lemma \ref{H-V-l}, we have the following result on the boundedness of $M$ on Herz spaces with variable exponent.
\begin{theorem}\label{H-V}
	Let $\alpha\in\bR$, $0<p<\infty$ and $q(\cdot):\bR^n\rightarrow (0,\infty)$ be a globally log-H{\"o}lder continuous function satisfying
	$1<q_{-}<q_{+}<\infty$. If $-\delta_{q(\cdot)}<\alpha<-\delta_{q'(\cdot)}$, where $\delta_{q(\cdot)}$ and $\delta_{q'(\cdot)}$ are constants appeared in (\ref{G-2-3}). then $M$ is bounded on $\dot{K}^{\alpha,p}_{q(\cdot)}(\bR^n)$.
\end{theorem}
The previous theorem recovers part of the results in [\cite{izuki2010boundedness}, Theorem 3.3], where the boundedness for a class of sublinear operators was obtained.

Theorem \ref{extrapo} and Theorem \ref{H-V} yield the following extrapolation theorem on Herz spaces with variable exponent.
\begin{theorem}\label{extrapo-var}
	Let  $\alpha\in\bR$, $0<s<p<\infty$, and $q(\cdot):\bR^n\rightarrow (0,\infty)$ be a globally log-H{\"o}lder continuous function such that
	$s<q_{-}<q_{+}<\infty$ for and $-n\delta_{{q(\cdot)/s}}<s\alpha<n\delta_{(q(\cdot)/s)'}$, where $\delta_{{q(\cdot)/s}}$ and $\delta_{(q(\cdot)/s)'}$ are constants appeared in (\ref{G-2-3}). Let
	$T$ be an operator satisfying, for any given $w\in A_1(\bR^n)$ and any $f\in L^s_w(\bR^n)$,
	\begin{equation*}
	\|T(f)\|_{L^s_w}\lesssim \|f\|_{L^s_w},
	\end{equation*}
	where the implicated constant is independent on $f$, but depends on $s$ and $[w]_{A_1}$. Then for any $f\in \dot{K}^{\alpha,p}_{q(\cdot)}(\bR^n)$,
	\begin{equation*}
	\|T(f)\|_{\dot{K}^{\alpha,p}_{q(\cdot)}}\lesssim \|f\|_{\dot{K}^{\alpha,p}_{q(\cdot)}}.
	\end{equation*}
\end{theorem}
\begin{proof}
	In view of $-n\delta_{{q(\cdot)/s}}<s\alpha<n\delta_{(q(\cdot)/s)'}$, we know that $M$ is bounded on $\dot{K}^{-s\alpha,(p/s)'}_{(q(\cdot)/s)'}(\bR^n)$ by using Theorem \ref{H-V}. Then using Theorem \ref{extrapo} and some basic properties of variable spaces mentioned above, we can finish the proof. The details is left to the interested readers.
\end{proof}
One can see that all the results in Section 5 hold for Herz spaces with variable exponent by using Theorem \ref{extrapo-var}.
\subsection{Mixed Herz spaces}
In this subsetion, the letter $\vec{q}$ denotes $n$-tuples of the numbers in $[0,\infty]$,~($n\geq2$),~$\vec{q}=(q_1,\cdots,q_n)$. By definition, the inequality, for example, $0<\vec{q}<\infty$ means $0<q_i<\infty$ for all $i$. For $1\leq\vec{q}\leq\infty$, we denote $\vec{q}'=(q'_1,\cdots,q'_n)$, where $q'_i$ satisfies
${1}/{q_i}+{1}/{q'_i}=1$.
We begin with recalling the notion of mixed-norm Lebesgue spaces.
\begin{definition}\label{d2-1}
	Let ~$\vec{q}=(q_1,\cdots,q_n)\in(0,\infty)^n$. Then the mixed-norm Lebesgue norm $\|\cdot\|_{L^{\vec{q}}}$ is defined by
	\begin{eqnarray*}
		\|f\|_{L^{\vec{q}}}
		:= \left(\int_{\bR}\cdots \left(\int_{\bR}\left(\int_{\bR}|f(x_1,x_2,\cdots,x_n)|^{q_1}dx_1\right)^{\frac{q_2}{q_1}}dx_2\right)^{\frac{q_3}{q_2}}\cdots dx_n\right)^{\frac{1}{q_n}}
	\end{eqnarray*}
	where $f: \bR^n \rightarrow \mathbb{C}$ is a measurable function. If $q_j=\infty$ for some $j=1,\cdots,n$, then we have to make appropriate modifications. We define the mixed-norm Lebesgue space $L^{\vec{q}}(\bR^n)$
	to be the set of all $f\in \mathcal{M}(\bR^n)$  with $\|f\|_{L^{\vec{q}}}<\infty$.
\end{definition}
The space $L^{\vec{q}}(\bR^n)$ was studied by Benedek and Panzone [\cite{1961The}] in 1961, which can be traced back to
H{\"o}rmander [\cite{hormander1960estimates}]. From the definition of $L^{\vec{q}}(\bR^n)$, it is easy to deduce that the mixed-norm Lebesgue
space $L^{\vec{q}}(\bR^n)$ is a ball quasi-Banach function space. If further $\vec{q}=(q_1,\cdots,q_n)\in[1,\infty)^n$, then $L^{\vec{q}}(\bR^n)$ is precisely a ball Banach function space, and in this case, the associate space of $L^{\vec{q}}(\bR^n)$ is its dual space $L^{\vec{q}'}(\bR^n)$ (see [\cite{1961The}]).

By substituting $X$ with $L^{\vec{q}}(\bR^n)$ in Definition \ref{D-2-1-1}, we recover the non-homogeneous mixed Herz space $\dot{K}^{\alpha,p}_{\vec{q}}(\bR^n)$ and homogeneous mixed Herz space ${K}^{\alpha,p}_{\vec{q}}(\bR^n)$ defined in [\cite{wei2021characterization,wei-2021}].

Obviously, $\delta_{\vec{q}}=\frac{1}{n}{\sum_{i=1}^n1/q_i}$. Noting that $M$ is bounded on $L^{\vec{q}}(\bR^n)$ for $1<\vec{q}<\infty$ (see [\cite{nogayama2019mixed}]), we have the boundedness of $M$ on $\dot{K}^{\alpha,p}_{\vec{q}}(\bR^n)$ as a consequence of Theorem \ref{T-3}:
\begin{theorem}\label{T-mH}
	Let $\alpha\in\bR$, $0<p<\infty, 1<\vec{q}<\infty$ such that $-\sum_{i=1}^n{1}/{q_i}<\alpha<\sum_{i=1}^n{1}/{q'_i}.$ Then the Hardy-Littlewood maximal operator $M$ is bounded on $\dot{K}^{\alpha,p}_{\vec{q}}(\bR^n)$.
\end{theorem}
Note that Theorem \ref{T-mH} is just the result of [\cite{wei2021extrapolation}, Theorem 4.2].

Combining Theorem \ref{extrapo} with Theorem \ref{T-mH}, we obtain the following extrapolation theorem on mixed Herz spaces.
\begin{theorem}\label{extrapo-mh}
	Let $\alpha\in\bR$, $0<s<p<\infty$, $s<\vec{q}<\infty$ such that $-\sum_{i=1}^n1/q_i<\alpha<n/s-\sum_{i=1}^n1/q_i$. Suppose
	$T$ is an operator satisfying, for any given $w\in A_1(\bR^n)$ and any $f\in L^s_w(\bR^n)$,
	\begin{equation*}
	\|T(f)\|_{L^s_w}\lesssim \|f\|_{L^s_w},
	\end{equation*}
	where the implicated constant is independent on $f$, but depends on $s$ and $[w]_{A_1}$. Then for any $f\in \dot{K}^{\alpha,p}_{\vec{q}}(\bR^n)$,
	\begin{equation*}
	\|T(f)\|_{\dot{K}^{\alpha,p}_{\vec{q}}}\lesssim \|f\|_{\dot{K}^{\alpha,p}_{\vec{q}}}.
	\end{equation*}
\end{theorem}
\begin{proof}
	Note that the condition $-\sum_{i=1}^n1/q_i<\alpha<n/s-\sum_{i=1}^n1/q_i$ is equivalent to $-\sum_{i=1}^n\frac{1}{({q_i}/s)'}<-s\alpha<\sum_{i=1}^n\frac{1}{{q_i}/s}$. Therefore, we get the boundedness of $M$ on $ \dot{K}^{-s\alpha,(p/s)'}_{(\vec{q}/s)'}(\bR^n)$ by using Theorem \ref{T-mH}. Our result follows directly from Theorem \ref{extrapo}.
\end{proof}
It is worthy mentioning that a more precise extrapolation theorem on mixed Herz spaces was obtained in [\cite{wei-2021}, Theorem 5.3], where the weighted inequality was  supposed to hold only for a subclass of $A_1(\bR^n)$ weights.
\subsection{Herz-Lorentz spaces}
We begin with the notions of Lorentz spaces.
\begin{definition}
	Given $f$ a  measurable function on $\bR^n$ and $0<r,q\leq\infty$, define
	\begin{equation*}
	\|f\|_{L^{r,q}}:=\left\{
	\begin{aligned}
	&\l(\int_{\bR^n}(t^{1/r}f^*(t))^q\frac{dt}{t}\r)^{1/q},&~q<\infty;\\
	& \sup_{t>0}t^{1/r}f^*(t), &~q=\infty;
	\end{aligned}
	\right.
	\end{equation*}
	where $f^*(y)=\inf\{s\in\bR:|\{t\in\bR^n:|g(t)|>s\}|\leq y\}$ is the non-increasing rearrangement of $f$ on $[0,\infty)$.
	The set of all $f$ with $\|f\|_{L^{r,q}}$ is denoted by $L^{r,q}(\bR^n)$ and is called the Lorentz space with indices $r$ and $q$.
\end{definition}
It was proved in [\cite{bennett1988interpolation}] that for $0<r,q\leq\infty$, $L^{r,q}(\bR^n)$ is a quasi-Banach function space. Moreover, if $1\leq r,q\leq\infty$, $L^{r,q}(\bR^n)$ is then a Banach function space. 

Following [\cite{grafakos2008classical}, Theorem 1.4.16], we know that for $1<r,q<\infty$, the dual space of $L^{r,q}(\bR^n)$ is $L^{r',q'}(\bR^n)$. Since simple functions are dense in $L^{r,q}(\bR^n)$ for $0<p,q<\infty$ (see [\cite{grafakos2008classical}, Theorem 1.4.13]), we know that $L^{r,q}(\bR^n)$ has an absolutely continuous norm.
As a consequence, the associate space of $L^{r,q}(\bR^n)$ is still $L^{r',q'}(\bR^n)$ by using [\cite{bennett1988interpolation}, Chapter 1, Corollary 4.3] for $1<r,q<\infty$.

The boundedness of $M$ on $L^{r,q}(\bR^n)$ was obtained in [\cite{arino1990maximal}, Theorem 1.7].
\begin{lemma}\label{BMHL}
	Let $1<r\leq\infty$, $1\leq q<\infty$. Then $M$ is bounded on $L^{r,q}(\bR^n)$.
\end{lemma}
Now we are in a position to define Herz-Lorentz spaces. Although the definition is obvious, we write them here since Herz-Lorentz spaces haven't appeared anywhere so far to the best of our knowledge. 
\begin{definition}\label{H_L}
	Let ~$\alpha\in \bR$, $0< p<\infty$ and $0<r,q\leq\infty$.
	
	{\rm (i)} The homogeneous Herz-Lorentz space $\dot{K}^{\alpha,p}_{r,q}(\bR^n)$ is defined by
	\begin{eqnarray*}
		\|f\|_{\dot{K}^{\alpha,p}_{r,q}}:=
		\l\{f\in L^{r,q}_{\rm{loc}}(\bR^n\backslash\{0\}): \|f\|_{\dot{K}^{\alpha,p}_{r,q}}<\infty\r\},
	\end{eqnarray*}
	where
	\begin{eqnarray*}
		\|f\|_{\dot{K}^{\alpha,p}_{r,q}}:=
		\l\{\sum_{k\in\mathbb{Z}}2^{kp\alpha}\|f\chi_{k}\|^p_{L^{r,q}}\r\}^{{1}/{p}}.
	\end{eqnarray*}
If $p=\infty$, then we have to make appropriate modifications.

	{\rm (ii)} The non-homogeneous Herz-Lorentz space $K^{\alpha,p}_{r,q}(\bR^n)$ is defined by
	\begin{eqnarray*}
		\|f\|_{K^{\alpha,p}_{r,q}}:=
		\l\{f\in L^{r,q}_{\rm{loc}}(\bR^n): \|f\|_{K^{\alpha,p}_{r,q}}<\infty\r\},
	\end{eqnarray*}
	where \begin{eqnarray*}
		\|f\|_{K^{\alpha,p}_{r,q}}:=
		\l\{\sum_{k\in\mathbb{N}}2^{kp\alpha}\|f\tilde{\chi}_k\|^p_{L^{r,q}}\r\}^{{1}/{p}}.
	\end{eqnarray*}
If $p=\infty$, then we have to make appropriate modifications.
\end{definition}

It was proved in [\cite{grafakos2008classical}, Example 1.4.8] that for any measurable set $E\subseteq \bR^n$, $\|\chi_E\|_{L^{r,q}}\sim |E|^{1/r}$ for $0<p,q<\infty$. Consequently, $\delta_{L^{r,q}}=1/r$. From this observation and Lemma \ref{BMHL}, we can prove the boundedness of $M$ on $\dot{K}^{\alpha,p}_{r,q}(\bR^n)$ by using Theorem \ref{T-3}.
\begin{theorem}\label{T-hl}
	Let $\alpha\in\bR$, $0<p<\infty$,  $1<r<\infty$ and $1\leq q<\infty$, such that $-n/r<\alpha<n/r'.$ Then the Hardy-Littlewood maximal operator $M$ is bounded on $\dot{K}^{\alpha,p}_{r,q}(\bR^n)$.
\end{theorem}
Combining Theorem \ref{extrapo} with Theorem \ref{T-hl}, and using the basic theories of Lorentz spaces mentioned above, we can establish the extrapolation theorem on Herz-Lorentz spaces as follows.
\begin{theorem}\label{T-hl2}
	Let $\alpha\in\bR$, $0<s<p<\infty$,  $s<r<\infty$ and $s< q<\infty$, such that $-n/r<\alpha<n(1/s-1/r)$. Suppose
	$T$ is an operator satisfying, for any given $w\in A_1(\bR^n)$ and any $f\in L^s_w(\bR^n)$,
	\begin{equation*}
	\|T(f)\|_{L^s_w}\lesssim \|f\|_{L^s_w},
	\end{equation*}
	where the implicated constant is independent on $f$, but depends on $s$ and $[w]_{A_1}$. Then for any $f\in \dot{K}^{\alpha,p}_{r,q}(\bR^n)$,
	\begin{equation*}
	\|T(f)\|_{\dot{K}^{\alpha,p}_{r,q}}\lesssim \|f\|_{\dot{K}^{\alpha,p}_{r,q}}.
	\end{equation*}
\end{theorem}
\begin{proof}
	Note that the condition $-n/r<\alpha<n(1/s-1/r)$ is equivalent to $-n\frac{1}{(r/s)'}<-s\alpha<n\frac{1}{r/s}$. Therefore, $M$ is bounded on 
	$\dot{K}^{-s\alpha,(p/s)'}_{(r/s)',(q/s)'}(\bR^n)$ by using Theorem \ref{T-hl}. Our result follows directly from Theorem \ref{extrapo}.
\end{proof}
By using Theorem \ref{T-hl2}, one can see that all the results in Section 5 hold for Herz-Lorentz spaces.

\subsection*{Declarations}
{\bf Conflict of interest}~  The author’s declare that they have no conflict of interest.

\subsection*{Data availability statement}
Our manuscript has no associated data.

\subsection*{Acknowledgements}
The authors would like to express their deep gratitude to the anonymous referees for their careful reading of the manuscript and their comments and suggestions. This work is supported by the National Natural Science Foundation of China (No. 11871452), the Natural Science Foundation of Henan Province (No. 202300410338) and the Nanhu Scholar Program for Young Scholars of Xinyang Normal University.\\


\end{document}